\documentclass[a4paper,11pt]{amsart}
\usepackage[utf8]{inputenc}
\usepackage{amssymb,amsmath,amsthm,mathtools,amsfonts}
\usepackage[margin=1in]{geometry}
\usepackage{dsfont}
\newtheorem{lemma}{Lemma}
\newtheorem{theorem}{Theorem}
\newtheorem*{remark}{Remark}
\title[A uniqueness lemma]{A uniqueness lemma with applications to regularization and incompressible fluid mechanics.}
\author{Guillaume L\'{e}vy$^{1}$}
\address{$^{1}$Laboratoire Jacques-Louis Lions, UMR 7598, Université Pierre
et Marie Curie, 75252 Paris Cedex 05, France.}
\email{$^{1}${levy@ljll.math.upmc.fr}}

\begin{document}

\begin{abstract}
 In this paper, we extend our previous result from \cite{NoteAuCRAS}.
 We prove that transport equations with rough coefficients do possess a uniqueness property, even in the presence of viscosity.
 Our method relies strongly on duality and bears a strong resemblance with the well-known DiPerna-Lions theory first developed in \cite{DiPerna-Lions}.
 This uniqueness result allows us to reprove the celebrated theorem of J. Serrin \cite{Serrin} in a novel way.
 As a byproduct of the techniques, we derive an $L^1$ bound for the vorticity in terms of a critical Lebesgue norm of the velocity field.
 We also show that the zero solution is unique for the 2D Euler equations on the torus under a mild integrability assumption.
 \textbf{TODO : chercher diverses \'equations classiques o\`u les id\'ees d'unicit\'e s'appliquent}
\end{abstract}

\maketitle

\section{Introduction}
In their seminal paper \cite{DiPerna-Lions}, R. J. DiPerna and P.-L. Lions proved the existence and uniqueness of solutions to transport equations on $\mathbb{R}^d$.
We recall here a slightly simplified version of their statement. 
\begin{theorem}[DiPerna-Lions]
 Let $d \geq 1$ be an integer.
 Let $1 \leq p \leq \infty$ and $p'$ its H\"older conjugate.
 Let $a_0$ be in $L^p(\mathbb{R}^d)$.
 Let $v$ be a fixed divergence free vector field in $L^1_{loc}(\mathbb{R}_+, \dot{W}^{1,p'}(\mathbb{R}^d))$.
 Then there exists a unique distributional solution $a$ in $L^{\infty}(\mathbb{R}_+,L^p(\mathbb{R}^d))$ of the Cauchy problem
 \begin{equation}
    \left \{
\begin{array}{c  c}
    \partial_t a + \nabla \cdot  (a v) = 0 \\
    a (0)  = a_0,  \\
\end{array}
\right.
 \end{equation}
 with the initial condition understood in the sense of $\mathcal{C}^0(\mathbb{R}_+, \mathcal{D}'(\mathbb{R}^d))$.
 We recall that $a$ is a distributional solution of the aforementioned Cauchy problem if and only if, 
 for any $\varphi$ belonging to $\mathcal{D}(\mathbb{R}_+ \times \mathbb{R}^d)$ and any $T > 0$, there holds
 \begin{equation}
  \int_0^T \int_{\mathbb{R}^d} a(t,x) \left(\partial_t \varphi(t,x) + v(t,x) \cdot \nabla \varphi(t,x) \right) dx dt 
  = \int_{\mathbb{R}^d} a(T,x) \varphi(T,x) dx - \int_{\mathbb{R}^d} a_0(x) \varphi(0,x) dx.
 \end{equation}

\end{theorem}
Beyond this theorem, many authors have since proved similar existence and (non-)uniqueness theorems, see for instance 
\cite{AmbrosioBV}, \cite{AmbrosioCrippa}, \cite{BouchutJames1}, \cite{BouchutJames2}, \cite{BouchutJames3}, \cite{Depauw}, \cite{LeBrisLions}, \cite{LeFlochXin}, 
\cite{Lerner} and references therein.
In particular, the papers \cite{BouchutJames1}, \cite{BouchutJames2} and \cite{BouchutJames3} use a duality method which is close in spirit to our results. 
Our key result, which relies on the maximum principle for the \emph{adjoint} equation, is both more general and more restrictive than the DiPerna-Lions theorem.
The generality comes from the wider range of exponents allowed, along with the affordability of additional scaling-invariant and/or dissipative terms in the equation.
We thus extend the result from \cite{NoteAuCRAS}, where the setting was restricted to the $L^2_{t,x}$ case and no right-hand side was considered. 
On the other hand, we do not fully extend the original theorem, since we are unable to prove the existence of solutions in the uniqueness classes.
Here is the statement.
\begin{theorem}
 Let $d \geq 1$ be an integer.
 Let $\nu \geq 0$ be a positive parameter.
 Let $1 \leq p,q \leq \infty$ be real numbers with Hölder conjugates $p'$ and $q'$.
 Let $v = v(t,x)$ be a fixed, divergence free vector field in $L^{p'}(\mathbb{R}_+, \dot{W}^{1,q'}(\mathbb{R}^d))$. 
 Given a time $T^* > 0$, let $a$ be in $L^p([0,T^*], L^q(\mathbb{R}^d))$. 
 Assume that $a$ is a distributional solution of the Cauchy problem
 \begin{equation}  
  (C) \left \{
\begin{array}{c  c}
    \partial_t a + \nabla \cdot  (a v) - \nu \Delta a = 0 \\
    a (0)  = 0,  \\
\end{array}
\right.
 \end{equation}
with the initial condition understood in the sense of $\mathcal{C}^0([0,T^*], \mathcal{D}'(\mathbb{R}^d))$.
That is, we assume that, for any function $\varphi$ in $\mathcal{D}(\mathbb{R}_+ \times \mathbb{R}^d)$ and any $T > 0$, there holds
\begin{equation}
 \int_{\mathbb{R}_+ \times \mathbb{R}^d} a(t,x) \left(\partial_t \varphi(t,x) + v(t,x) \cdot \nabla \varphi(t,x) + \nu \Delta \varphi(t,x) \right) dx dt 
 = \int_{\mathbb{R}^d} u(T,x) \varphi(T,x) dx.
\end{equation}
Then $a$ is identically zero on $[0,T^*] \times \mathbb{R}^d$.
\label{Uniqueness}
\end{theorem}
Though one may fear that the lack of existence might render the theorem unapplicable in practice, it does not.
For instance, when working with the Navier-Stokes equations, the vorticity of a Leray solution only belongs, a priori, to 
$$L^{\infty}(\mathbb{R}_+, \dot{H}^{-1}(\mathbb{R}^d)) \cap L^2(\mathbb{R}_+ \times \mathbb{R}^d).$$
In particular, the only Lebesgue-type space to which this vorticity belongs is $L^2(\mathbb{R}_+ \times \mathbb{R}^d)$.
Our theorem is well suited for solutions possessing \emph{a priori} no integrable derivative whatsoever.

As such, our theorem appears a regularization tool.
The philosophy is that, if an equation has smooth solutions, then any sufficiently integrable \emph{weak} solution is automatically smooth.
We illustrate our theorem with an application to the regularity result of J. Serrin \cite{Serrin} and subsequent authors 
\cite{BeiraoDaVeiga}, \cite{CaffKohnNiren}, \cite{CheminZhang}, \cite{FabesJonesRiviere}, \cite{FabreLebeau}, \cite{Giga}, \cite{IskauSereginSverak}, \cite{Struwe}, \cite{vonWahl}.

 We warn the reader that we did \emph{not} prove that the Leray solutions are unique in their class and will not claim so.
 Indeed, the uniqueness stated in Theorem \ref{UniquenessSerrin} is purely linear.
 In particular, it does not use the link between the vorticity and the exterior fields.
 It does not rely either on the divergence freeness of the vorticity.
 The key point in our proof is the maximum principle of the \emph{adjoint equation}.
 The validity of the maximum principle partially depends on the vorticity equation having only differential operators rather than pseudodifferential ones.
 
 Another standpoint on this theorem, which we owe to a private communication from N. Masmoudi, is that we now have two ways to recover the vorticity field $\Omega$ from the velocity.
 We may either we use the defining identity 
 $$\Omega := \nabla \wedge u$$ 
 or that $\Omega$ is the unique solution of the linear problem
 $$ 
  (NSV) \left \{
\begin{array}{c  c}
    \partial_t \Omega + \nabla \cdot  (\Omega \otimes u) - \Delta \Omega = \nabla \cdot ( u \otimes \Omega)\\
    \Omega (0)  = \nabla \wedge u_0.  \\
\end{array}
\right.
$$
 The second choice makes a strong use of the peculiar algebra of the Navier-Stokes equations, while the first one is general and requires no other assumption on $u$ than the divergence-free condition.
 Thus, we may hope to garner more information from the vorticity uniqueness, even though it may seem circuitous.
 Embodied by Theorem \ref{ExistenceSerrin} is our new approach to the Serrin-type regularity results, relying on finer algebraic properties of the equation than
 its belonging to the semilinear heat equations family.
 
 \section{Results}
 Let us comment a bit on the strategy we shall use. 
First, because $a$ lies in a low-regularity class of distributions, energy-type estimates seem out of reach.
Thus, a duality argument is much more adapted to our situation.
Given the assumptions on $a$, which for instance imply that $\Delta a$ is in $L^p(\mathbb{R}_+, \dot{W}^{-2, q}(\mathbb{R}^d))$, we need to prove the following existence result.
\begin{theorem}
 Let $\nu \geq 0$ be a positive real number.
 Let $v = v(t,x)$ be a fixed, divergence free vector field in $L^{p'}(\mathbb{R}_+, \dot{W}^{1,q'}(\mathbb{R}^d))$. 
 Let $\varphi_0$ be a smooth, compactly supported function in $\mathbb{R}^d$.
 There exists a function $\varphi$ in $L^{\infty}(\mathbb{R}_+ \times \mathbb{R}^d)$ solving
 \begin{equation}
  (C') \left \{
\begin{array}{c  c}
    \partial_t \varphi -   \nabla \cdot ( \varphi v) - \nu \Delta \varphi = 0 \\
    \varphi (0)  = \varphi_0  \\
\end{array}
\right.
 \end{equation}
 in the sense of distributions and satisfying the estimate
 $$\|\varphi(t)\|_{L^{\infty}(\mathbb{R}^d)} \leq \|\varphi_0\|_{L^{\infty}(\mathbb{R}^d)}. $$
\label{Existence}
\end{theorem}
Picking some positive time $T > 0$ and considering $\varphi(T - \cdot)$ instead of $\varphi$, 
Theorem \ref{Existence} amounts to build, for $T > 0$, a solution on $[0,T] \times \mathbb{R}^d$ of the Cauchy problem
\begin{equation}
  (-C') \left \{
\begin{array}{c  c}
    - \partial_t \varphi -   \nabla\cdot( \varphi v) - \nu\Delta \varphi = 0 \\
    \varphi (T)  = \varphi_0 . \\
\end{array}
\right.
 \end{equation}
 This theorem is a slight generalization of the analogue theorem in the Note \cite{NoteAuCRAS}.
 The proof we provide here follows the same lines but retains only the key estimate, which is the boundedness of the solution.
 The additional estimate in the Note was inessential and had the inconvenient to degenerate when the viscosity coefficient is small.
 In contrast, the boundedness is unaffected by such changes. 
 The techniques used in the proof of Theorem \ref{Existence} are robust.
 This robustness is encouraging for future work, as many generalizations are possible depending on the needs.
 We will not try to list them all ; instead, we give some examples of possible adaptations to other contexts.
 The most direct one is its analogue for diagonal systems, for uniqueness in this case reduces to applying the scalar case to each component of the solution.
 Alternatively, one may add various linear, scaling invariant terms on the right hand side, or any dissipative term (such as a fractional laplacian) on the left hand side.
 Also, in view of application to compressible fluid mechanics, the main theorems remain true without the divergence freeness of the transport field
 provided that the negative part of its divergence belongs to $L^1(\mathbb{R}_+, L^{\infty}(\mathbb{R}^d))$.
 This extension was already present in the original paper \cite{DiPerna-Lions} from R.J. DiPerna and P.-L. Lions.
 
 Among these numerous variants, a particular one stands out.
 It applies to a restricted family of equations, which are essentially the Navier-Stokes equations with frozen coefficients.
 These equations are obtained from $(C)$ by adding a linear, non diagonal term on the right-hand side, of a peculiar form.
 The purpose of this variant is to provide a different proof of the renowned Serrin theorem.
 We now state it.
 \begin{theorem}
 Let $d \geq 3$ be an integer.
 Let $\nu > 0$ be a positive real number.
 Let $2 \leq p < \infty$ and $d < q \leq \infty$ be real numbers satisfying $\frac 2p + \frac dq = 1$.
 Let $v$ be a fixed divergence free vector field in $L^2(\mathbb{R}_+, \dot{H}^1(\mathbb{R}^d))$.
 Let $w$ be a fixed vector field in $L^2(\mathbb{R}_+, \dot{H}^1(\mathbb{R}^d)) \cap L^p(\mathbb{R}_+, L^q(\mathbb{R}^d))$.
 Let $a$ be in $L^2(\mathbb{R}_+ \times \mathbb{R}^d)$. 
 Assume that $a$ is a distributional solution of the Cauchy problem
 \begin{equation}  
  (C_{NS}) \left \{
\begin{array}{c  c}
    \partial_t a + \nabla \cdot  (a \otimes v) - \nu \Delta a =  \nabla \cdot (w \otimes a) \\
    a (0)  = 0,  \\
\end{array}
\right.
 \end{equation}
with the initial condition understood in the sense of $\mathcal{C}^0([0,T], \mathcal{D}'(\mathbb{R}^d))$.
Then $a$ is identically zero on $\mathbb{R}_+ \times \mathbb{R}^d$.
\label{UniquenessSerrin}
\end{theorem}
This time, the addition of a non diagonal -- though scaling invariant -- term induces some notable changes, because of two algebraic facts which we wish to emphasize.
The first one relates to the divergence freeness of the solution when dealing with the Navier-Stokes equations.
Indeed, we have in this case the equality
$$\nabla \cdot (w \otimes \Omega) = \Omega \cdot \nabla w $$
 and both sides make sense as distributions.
 However, since we forget the divergence freeness of $a$ when we compute the adjoint equation, 
 it is of utmost importance to write the equation with the right-hand side written in its divergence form $\nabla \cdot (w \otimes a)$.
 This divergence form is the only one with which we are able to get a essential bound (or generalized maximum principle) for the adjoint equation, an absolutely crucial feature of our proof.
 The second one stems from the vectorial nature of the solution $a$, which complexifies the integration by parts of the term $|a|^{r-2}a \Delta a$.
 As it is well-known, adding a laplacian term in a partial differential equation has a smoothing effect on solutions.
 However, when $r$ grows, the smoothing effect concentrates mostly on $|a|^2$ and not on the full solution $a$.
 While this may look like a trivial observation to the accustomed reader, it is precisely what prevents us from removing the scale-invariant assumption that $w$ belongs to 
 $L^p(\mathbb{R}_+, L^q(\mathbb{R}^d))$.
 If we were able to lift it -- which we believe we cannot, owing to the numerical results of J. Guillod and V. V. \v{S}ver\'ak in \cite{GuillodSverak} --, 
 then a linear uniqueness statement for Leray solutions would hold.
 
 \begin{remark}
  Theorem \ref{UniquenessSerrin} also holds in the limit case $(p,q) = (d,\infty)$, provided that $w$ satisfies the smallness condition
  $$\|w\|_{L^{\infty}(\mathbb{R}_+, L^d(\mathbb{R}^d))} < \frac{2\nu}{C},$$
  where $C$ is the Sobolev constant associated to the embedding $\dot{H}^1(\mathbb{R}^d) \hookrightarrow L^{\frac{2d}{d-2}}(\mathbb{R}^d)$.
 \end{remark}

 To prove Theorem \ref{UniquenessSerrin}, we will need, as for Theorem \ref{Uniqueness}, a dual existence result, which we state.
 \begin{theorem}
 Let $d \geq 3$ be an integer.
 Let $\nu > 0$ be a positive real number.
 Let $2 \leq p < \infty$ and $d < q \leq \infty$ be real numbers satisfying $\frac 2p + \frac dq = 1$.
 Let $v$ be a fixed divergence free vector field in $L^2(\mathbb{R}_+, \dot{H}^1(\mathbb{R}^d))$.
 Let $w$ be a fixed vector field in $L^2(\mathbb{R}_+, \dot{H}^1(\mathbb{R}^d)) \cap L^p(\mathbb{R}_+, L^q(\mathbb{R}^d))$.
 There exists a solution $\varphi$ to the following Cauchy problem
 \begin{equation}  
  (C'_{NS}) \left \{
\begin{array}{c  c}
    \partial_t \varphi - \nabla \cdot  (\varphi \otimes v) - \nu \Delta \varphi =  - ^t \nabla \varphi \cdot a \\
    \varphi (0)  = \varphi_0 \in \mathcal{D}(\mathbb{R}^d)  \\
\end{array}
\right.
 \end{equation}
satisfying in addition, for almost every $t > 0$,
\begin{equation}
 \|\varphi(t)\|_{L^{\infty}(\mathbb{R}^d)} \leq \|\varphi_0\|_{L^{\infty}(\mathbb{R}^d)} \exp\left[\frac{C^p}{p\nu^{p-2}} \int_0^t \|w(s)\|_{L^q(\mathbb{R}^d)}^p ds \right]. 
 \label{GrowingMaximumPriciple}
\end{equation}
Above, $C$ denotes a constant depending only on the dimension $d$.
 \label{ExistenceSerrin} 
 \end{theorem}
In the right hand side of the main equation, the quantity $ - ^t \nabla \varphi \cdot a$ is a shorthand for
$$ - \nabla (\varphi \cdot a) +  ^t \nabla a \cdot \varphi $$
and this last expression makes sense in $L^2(\mathbb{R}_+, \dot{H}^{-1}_{loc}(\mathbb{R}^d)) + L^2(\mathbb{R}_+ \times \mathbb{R}^d)$ provided that $\varphi$ is bounded in space-time.
Using coordinates, the different terms expand respectively as
$$(^t \nabla \varphi \cdot a)_i = \sum_{j=1}^d \partial_i \varphi_j a_j ; $$
$$(\nabla (\varphi \cdot a))_i = \sum_{j=1}^d  \partial_i(\varphi_j a_j) ; $$
$$(^t \nabla a \cdot \varphi)_i =  \sum_{j=1}^d \partial_i a_j \varphi_j.$$

Although the left-hand sides of $(C_{NS})$ and its adjoint equation $(C'_{NS})$ are almost identical, their right-hand sides are different.
This discrepancy has striking consequences on their global behaviour, in that $(C'_{NS})$ does possess a generalized maximum principle, while $(C_{NS})$ does not.
That fact is the core of our paper, without which no conclusion on the Navier-Stokes and Euler equations could have been drawn.
Conversely, we are able to prove a uniqueness result for $(C_{NS})$ while we do not expect any analogous result for $(C'_{NS})$, at least at the present time.

As a consequence of Theorem \ref{UniquenessSerrin}, we give an alternative proof of the Serrin theorem in most cases.
This new proof has the advantage of making a stronger use of the algebra of the Navier-Stokes equations than the previous one.
To avoid technical details which would only obscure the proof, we choose to present it in the case of the three dimensional torus.
An analogue exists when the regularity assumption is written on the whole space $\mathbb{R}^3$, or a subdomain thereof, with a similar proof and some minor adjustments.
We recall the theorem of J. Serrin in its improved form by Y. Giga in \cite{Giga}, written with integrability assumptions on the Leray solution.
\begin{theorem}[J. Serrin]
 Let $u = u(t,x)$ be a Leray solution of the Navier-Stokes equations 
$$
  (NS) \left \{
\begin{array}{c  c}
    \partial_t u + \nabla \cdot  (u \otimes u) - \Delta u = - \nabla p \\
    \text{div }u = 0 \\
    u (0)  = u_0 \in L^2(\mathbb{T}^3)  \\
\end{array}
\right.
$$
 on $\mathbb{R}_+ \times \mathbb{T}^3$.
 Assume the existence of times $T_2 > T_1 > 0$ and exponents $2 \leq p < \infty, 3 < q \leq \infty$ such that $u$ belongs to $L^p(]T_1,T_2[, L^q(\mathbb{T}^3))$.
 Then $u$ belongs to $\mathcal{C}^{\infty}(]T_1,T_2[ \times \mathbb{T}^3)$.
 \label{NouvellePreuveSerrin}
\end{theorem}
 Besides reproving in a novel way the results of J. Serrin and his continuators, an immediate corollary of Theorem \ref{ExistenceSerrin} is the following.

  \begin{theorem}
   Let $d \geq 3$ be an integer.
   Let $\nu$ and $T$ be strictly positive real numbers.
   Let $u$ be a strong solution of the Navier-Stokes equations
 $$
    \left \{
 \begin{array}{c  c}
     \partial_t u + \nabla \cdot  (u \otimes u) - \nu \Delta u =  - \nabla p \\
     \text{div }u = 0 \\
 \end{array}
 \right.
 $$
 on $]0,T[ \times \mathbb{R}^d$.
 Then, there exists a constant $C$ depending only on $d$ such that for any $0 <  t < T$ and any $2 \leq p < \infty$, $d < q \leq \infty$ satisfying $\frac 2p + \frac dq = 1$, there holds
 \begin{equation}
  \|\Omega(t)\|_{L^1(\mathbb{R}^d)} \leq \|\Omega(0)\|_{L^1(\mathbb{R}^d)} \exp\left[\frac{C^p}{p\nu^{p-2}} \int_0^t \|u(s)\|_{L^q(\mathbb{R}^d)}^p ds \right].
 \end{equation}
   \label{BorneL1VorticiteForte}
  \end{theorem}
  
  Finally, applying Theorem \ref{Uniqueness} to the 2D Euler equations on the torus, one gets the following statement.
  
  \begin{theorem}
  Let $p \geq  2$ be a real number.
  Let $u$ be a weak solution of the Euler equations starting from zero initial data
   $$
    \left \{
 \begin{array}{c  c}
     \partial_t u + \nabla \cdot  (u \otimes u) =  - \nabla p \\
     \text{div }u = 0 \\
     u(0) = 0
 \end{array}
 \right.
 $$
 and assume that $\omega := \text{curl }u$ belongs to $L^{\infty}(\mathbb{R}_+, L^p(\mathbb{T}^2))$.
 Then $u$ is identically zero on $\mathbb{R}_+ \times \mathbb{T}^2$.
   \label{UniquenessZeroEuler}
  \end{theorem}

\section{Proofs}
 We state here a commutator lemma, similar to Lemma II.1 in \cite{DiPerna-Lions}, which we will use in the proof of Theorem \ref{Uniqueness}.
 \begin{lemma}
 Let $T > 0$.
 Let $v$ be a fixed, divergence free vector field in $L^{p'}(\mathbb{R}_+, \dot{W}^{1,q'}(\mathbb{R}^d))$.
 Let $a$ be a fixed function in $L^p(\mathbb{R}_+, L^q (\mathbb{R}^d))$.
 Let $\rho = \rho(x)$ be some smooth, positive and compactly supported function on $\mathbb{R}^d$.
 Normalize $\rho$ to have unit norm in $L^1(\mathbb{R}^d)$ and define $\rho_{\varepsilon} := \varepsilon^{-d} \rho\left(\frac{\cdot}{\varepsilon}\right)$.
 Define the commutator $C^{\varepsilon}$ by 
$$
 C^{\varepsilon}(t,x) :=  v(t,x) \cdot (\nabla \rho_{\varepsilon} \ast a(t))(x) - (\nabla \rho_{\varepsilon} \ast (v(t) a(t)))(x) .
$$
 Then, as $\varepsilon \to 0$,
 $$\|C^{\varepsilon}\|_{L^1(\mathbb{R}_+ \times \mathbb{R}^d)} \to 0. $$
 \label{LemmeCommutateur}
\end{lemma}

\begin{proof}
 For almost all $(t,x)$ in $\mathbb{R}_+ \times \mathbb{R}^d$, we have
 $$C^{\varepsilon}(t,x) = \int_{\mathbb{R}^d} \frac{1}{\varepsilon^d}a(t,y) \frac{v(t,x)-v(t,y)}{\varepsilon} \cdot \nabla \rho\left(\frac{x-y}{\varepsilon}\right) dy.$$
 Performing the change of variable $y = x + \varepsilon z$ yields
 $$C^{\varepsilon}(t,x) = \int_{\mathbb{R}^d} a(t,x+\varepsilon z) \frac{v(t,x)-v(t,x+\varepsilon z)}{\varepsilon} \cdot \nabla \rho(z) dz.$$
 Using the Taylor formula
 $$v(\cdot,x+\varepsilon z) - v(\cdot,x) = \int_0^1 \nabla v(\cdot,x+r\varepsilon z) \cdot (\varepsilon z) dr, $$
 which is true for smooth functions and extends to $\dot{W}^{1,q'}(\mathbb{R}^d)$ thanks to the continuity of both sides on this space 
 and owing to Fubini's theorem to exchange integrals, we get the nicer formula
 $$C^{\varepsilon}(t,x) = - \int_0^1 \int_{\mathbb{R}^d} a(t,x+\varepsilon z) \nabla v(t,x+r\varepsilon z) : ( \nabla \rho(z) \otimes z) dz dr,$$
 where $:$ denotes the contraction of rank two tensors.
 Because $q$ and $q'$ are dual H\"older exponents, at least one of them is finite.
 We assume for instance that $q < \infty$, the case $q' < \infty$ being completely similar.
 
 Let 
 $$\widetilde{C}^{\varepsilon}(t,x) := - \int_0^1 \int_{\mathbb{R}^d} a(t,x+r\varepsilon z) \nabla v(t,x+r\varepsilon z) : ( \nabla \rho(z) \otimes z) dz dr. $$
 We claim that, as $\varepsilon \to 0$,
 $$\|C^{\varepsilon}-\widetilde{C}^{\varepsilon}\|_{L^1(\mathbb{R}_+ \times \mathbb{R}^d)} \to 0.$$
 Integrating both in space and time and owing to H\"older's inequality, we have
 \begin{multline*}
   \|C^{\varepsilon}-\widetilde{C}^{\varepsilon}\|_{L^1(\mathbb{R}_+ \times \mathbb{R}^d)} \leq \\
 \int_0^1 \int_{\mathbb{R}^d} \int_0^{\infty} 
 \|a(t,\cdot+\varepsilon z)-a(t,\cdot+r\varepsilon z)\|_{L^q(\mathbb{R}^d)} \|\nabla v(t)\|_{L^{q'}(\mathbb{R}^d)} |\nabla \rho(z) \otimes z|  dt dz dr.
 \end{multline*}
Since $a \in L^p(\mathbb{R}_+, L^q(\mathbb{R}^d))$ and $q < \infty$, for almost any $t \in \mathbb{R}_+$, for all $z \in \mathbb{R}^d$ and $r \in [0,1]$, 
$$\|a(t,\cdot+\varepsilon z)-a(t,\cdot+r\varepsilon z)\|_{L^q(\mathbb{R}^d)} \to 0 $$
as $\varepsilon \to 0$.
Thanks to the uniform bound
\begin{multline*}
\|a(t,\cdot+\varepsilon z)-a(t,\cdot+r\varepsilon z)\|_{L^q(\mathbb{R}^d)}\|\nabla v(t)\|_{L^{q'}(\mathbb{R}^d)} |\nabla \rho(z) \otimes z| \leq \\
2 \|a(t)\|_{L^q(\mathbb{R}^d)}\|\nabla v(t)\|_{L^{q'}(\mathbb{R}^d)} |\nabla \rho(z) \otimes z|,
\end{multline*}
we may invoke the dominated convergence theorem to get the desired claim.

From this point on, we denote by $U(t,x)$ the quantity $a(t,x) \nabla v(t,x)$.
We notice that $U$ is a fixed function in $L^1(\mathbb{R}_+ \times \mathbb{R}^d)$ and that, by definition,
$$\widetilde{C}^{\varepsilon}(t,x) = - \int_0^1 \int_{\mathbb{R}^d} U(t,x+r\varepsilon z) : ( \nabla \rho(z) \otimes z) dz dr.$$
The normalization on $\rho$ yields the identity
$$- \int_{\mathbb{R}^d}  \nabla \rho(z)\otimes z  dz = \left(\int_{\mathbb{R}^d} \rho(z) dz\right) I_d = I_d, $$
where $I_d$ is the $d-$dimensional identity matrix.
This identity in turn entails that
$$\widetilde{C}^0(t,x) = a(t,x) \nabla v(t,x) : I_d = a(t,x) \text{ div }v(t,x) = 0. $$
A second application of the dominated convergence theorem to the function $U$ gives
$$\|\widetilde{C}^{\varepsilon} - \widetilde{C}^0\|_{L^1(\mathbb{R}_+ \times \mathbb{R}^d)} \to 0 $$
as $\varepsilon \to 0$, from which the lemma follows.
\end{proof}

 \begin{proof}[Proof of Theorem \ref{Existence}]
Let us choose some mollifying kernel $\rho = \rho(x)$ and denote $v_{\delta} := \rho_{\delta} \ast v$, where $\rho_{\delta}(x) := \delta^{-d} \rho(\frac{x}{\delta})$. 
Let $(C'_{\delta})$ be the Cauchy problem $(C')$ where we replaced $v$ by $v_{\delta}$.
The existence of a (smooth) solution $\varphi^{\delta}$ to $(C'_{\delta})$ is then easily obtained thanks to, for instance, a Friedrichs method combined with heat kernel estimates.
We now turn to the $L^{\infty}$ bound uniform in $\delta$.

Let $r \geq 2$ be a real number.
Multiplying the equation on $\varphi^{\delta}$ by $\varphi^{\delta} |\varphi^{\delta}|^{r-2}$ and integrating in space and time, we get
$$
 \frac 1r \|\varphi^{\delta}(t)\|_{L^r(\mathbb{R}^d)}^r + (r-1) \int_0^t \| \nabla \varphi^{\delta}(s) |\varphi^{\delta}(s)|^{\frac{r-2}{2}} \|_{L^2(\mathbb{R}^d)}^2 ds 
 = \frac 1r \|\varphi_0\|_{L^r(\mathbb{R}^d)}^r.
$$
Discarding the gradient term, taking $r$-th root in both sides and letting $r$ go to infinity gives
\begin{equation}
 \|\varphi^{\delta}(t)\|_{L^{\infty}(\mathbb{R}^d)} \leq \|\varphi_0\|_{L^{\infty}(\mathbb{R}^d)}.
 \label{BorneInfinie}
\end{equation}
Thus, the family $(\varphi^{\delta})_{\delta}$ is bounded in $L^{\infty}(\mathbb{R}_+ \times \mathbb{R}^d)$.
Up to an extraction, $(\varphi^{\delta})_{\delta}$ converges weak$-\ast$ in $L^{\infty}(\mathbb{R}_+ \times \mathbb{R}^d)$ to some function $\varphi$.

As a consequence, because $v_{\delta} \to v$ strongly in $L^1_{loc}(\mathbb{R}_+ \times \mathbb{R}^d)$ as $\delta \to 0$, the following convergences hold :
$$\Delta \varphi^{\delta} \rightharpoonup^{\ast} \Delta \varphi \text{ in } L^{\infty}(\mathbb{R}_+, \dot{W}^{-2,\infty} (\mathbb{R}^d)) ;$$ 
$$ \varphi^{\delta}v^{\delta} \rightharpoonup \varphi v \text{ in } L^1_{loc}(\mathbb{R}_+\times \mathbb{R}^d). $$ 
In particular, such a $\varphi$ is a distributional solution of $(C')$ with the desired regularity.
 \end{proof}
 We are now in position to prove the main theorem of this paper.
 \begin{proof}[Proof of Theorem \ref{Uniqueness}]
   Let $\rho = \rho(x)$ be a radial mollifying kernel and define $\rho_{\varepsilon}(x) := \varepsilon^{-d} \rho(\frac{x}{\varepsilon})$.
 Convolving the equation on $a$ by $\rho_{\varepsilon}$ gives, denoting $a_{\varepsilon} := \rho_{\varepsilon} \ast a$,
$$
  (C_{\varepsilon}) \ \ \partial_t a_{\varepsilon} + \nabla \cdot (a_{\varepsilon} v) - \nu\Delta a_{\varepsilon} 
  = C^{\varepsilon},
$$
 where the commutator $C^{\varepsilon}$ has been defined in Lemma \ref{LemmeCommutateur}.
Notice that even without any smoothing in time, $a_{\varepsilon}$, $\partial_t a_{\varepsilon}$ lie respectively in $L^{\infty}(\mathbb{R}_+, \mathcal{C}^{\infty}(\mathbb{R}^d))$ and 
$L^1(\mathbb{R}_+, \mathcal{C}^{\infty}(\mathbb{R}^d))$, which is enough to make the upcoming computations rigorous.
In what follows, we let $\varphi^{\delta}$ be a solution of the Cauchy problem $(-C'_{\delta})$, where  $(-C'_{\delta})$ is $(-C')$ (defined in Theorem \ref{ExistenceSerrin}) with $v$ replaced by $v_{\delta}$.
Let us now multiply, for $\delta, \varepsilon > 0$ the equation $(C_{\varepsilon})$ by $\varphi^{\delta}$ and integrate in space and time.
After integrating by parts (which is justified by the high regularity of the terms we have written), we get
$$
 \int_0^T \int_{\mathbb{R}^d} \partial_t a_{\varepsilon}(s,x) \varphi^{\delta}(s,x) dx ds = 
 \langle a_{\varepsilon}(T), \varphi_0 \rangle_{\mathcal{D}'(\mathbb{R}^d), \mathcal{D}(\mathbb{R}^d)}  - 
 \int_0^T \int_{\mathbb{R}^d} a_{\varepsilon}(s,x)  \partial_t \varphi^{\delta}(s,x) dx ds.
$$
From this identity, it follows that
\begin{multline*}
 \langle a_{\varepsilon}(T), \varphi_0 \rangle_{\mathcal{D}'(\mathbb{R}^d), \mathcal{D}(\mathbb{R}^d)}  
 = \int_0^T \int_{\mathbb{R}^d} \varphi^{\delta}(s,x) C^{\varepsilon}(s,x) dx ds \\ - 
 \int_0^T \int_{\mathbb{R}^d} a_{\varepsilon}(s,x) \left(- \partial_t \varphi^{\delta}(s,x) -  \nabla \cdot (v(s,x)\varphi^{\delta}(s,x)) - \nu\Delta \varphi^{\delta}(s,x) \right) dx ds.
\end{multline*}
From Lemma \ref{LemmeCommutateur}, we know in particular that $C^{\varepsilon}$ belongs to $L^1(\mathbb{R}_+ \times \mathbb{R}^d)$ for each fixed $\varepsilon > 0$.
Thus, in the limit $\delta \to 0$, we have, for each $\varepsilon > 0$,
$$\int_0^T \int_{\mathbb{R}^d} \varphi^{\delta}(s,x) C^{\varepsilon}(s,x) dx ds \to \int_0^T \int_{\mathbb{R}^d} \varphi(s,x) C^{\varepsilon}(s,x) dx ds. $$
On the other hand, the definition of $\varphi^{\delta}$ gives
$$- \partial_t \varphi^{\delta} - \nabla \cdot (v \varphi^{\delta}) - \nu\Delta \varphi^{\delta} = \nabla \cdot ((v_{\delta} - v) \varphi^{\delta}). $$
Thus, the last integral in the above equation may be rewritten, integrating by parts,
$$- \int_0^T \int_{\mathbb{R}^d} \varphi^{\delta} (v_{\delta} - v)\cdot  \nabla a_{\varepsilon}(s,x) dx ds. $$
For each fixed $\varepsilon$, the assumption on $a$ entails that $\nabla a_{\varepsilon}$ belongs to $L^p(\mathbb{R}_+,L^q(\mathbb{R}^d))$.
Furthermore, it is an easy exercise to show that
$$\|v_{\delta} - v\|_{L^{p'}(\mathbb{R}_+,L^{q'}(\mathbb{R}^d))} \leq 
\delta \|\nabla v\|_{L^{p'}(\mathbb{R}_+,L^{q'}(\mathbb{R}^d))} \||\cdot|\rho\|_{L^1(\mathbb{R}^d)}. $$
Now, taking the limit $\delta \to 0$ while keeping $\varepsilon > 0$ fixed, we have
$$
  \langle a_{\varepsilon}(T), \varphi_0 \rangle_{\mathcal{D}'(\mathbb{R}^d), \mathcal{D}(\mathbb{R}^d)} = \int_0^T \int_{\mathbb{R}^d} \varphi(s,x) C^{\varepsilon}(s,x) dx ds .
$$
Taking the limit $\varepsilon \to 0$ and using Lemma \ref{LemmeCommutateur}, we finally obtain
$$
   \langle a(T), \varphi_0 \rangle_{\mathcal{D}'(\mathbb{R}^d), \mathcal{D}(\mathbb{R}^d)} = 0.
$$
This being true for any test function $\varphi_0$, $a(T)$ is the zero distribution and finally $a \equiv 0$.
 \end{proof}

 \begin{proof}[Proof of Theorem \ref{ExistenceSerrin}]
  The proof of this Theorem is very similar to that of Theorem 3.1 in \cite{Struwe}.
  We nevertheless reproduce it in our cse for the sake of completeness.
  For simplicity, we reduce to the case $\nu = 1$. 
  Let $\rho = \rho(x)$ be a radial mollifying kernel and let us denote $\rho_{\delta}(x) = \delta^{-d}\rho(\frac{x}{\delta})$.
  Let $w_{\delta} = \rho_{\delta} \ast w$ and $v_{\delta} = \rho_{\delta} \ast v$.
  Let $(C'_{\delta})$ be the Cauchy problem $(C')$ with $w,v$ replaced by $w_{\delta}, v_{\delta}$ respectively.
  The existence of a smooth solution $\varphi^{\delta}$ to the Cauchy problem $(C'_{\delta})$ is easy and thus omitted.
  We focus on the the relevant estimates.
  Let $r \geq 2$ be a real number.
  We first take the scalr product of the equation on $\varphi^{\delta}$ by $\varphi^{\delta}$ and carefully rearrange the laplacian term to get the following equation on $|\varphi^{\delta}|^2$
  \begin{equation}
   \frac 12 \partial|\varphi^{\delta}|^2 + \frac 12 v_{\delta} \cdot \nabla |\varphi^{\delta}|^2 - \frac{\nu}{2} \Delta |\varphi^{\delta}|^2 + \nu|\nabla \varphi^{\delta}|^2
   = \varphi^{\delta} \cdot (^t \nabla \varphi^{\delta} \cdot w_{\delta}).
  \end{equation}
  For notational convenience, we let $\psi^{(\delta)} := |\varphi^{\delta}|^2$ in the sequel.
  Now, multiplying this new equation by $|\varphi^{\delta}|^{r-2}$ and integrating in space and time, we get 
    \begin{multline*}
 \frac 1r \|\varphi^{\delta}(t)\|_{L^r(\mathbb{R}^d)}^r + \frac{r-2}{4}\nu\int_0^t \| (\psi^{(\delta)})^{\frac{r-4}{4}} \nabla \psi^{(\delta)}\|_{L^2(\mathbb{R}^d)}^2  ds +  \nu \int_0^t \| \nabla \varphi^{\delta}(s) |\varphi^{\delta}(s)|^{\frac{r-2}{2}} \|_{L^2(\mathbb{R}^d)}^2 ds \\
 = \frac 1r \|\varphi_0\|_{L^r(\mathbb{R}^d)}^r - \int_0^t \int_{\mathbb{R}^d}|\varphi^{\delta}|^{r-2} \varphi^{\delta} \cdot (^t \nabla \varphi^{\delta} \cdot w_{\delta}) dx ds .
\end{multline*}
Denote by $I(t)$ the integral on the right hand side.
Rewriting 
$$I(t) = \int_0^t \int_{\mathbb{R}^d}(\psi^{(\delta)})^{\frac r4} |\varphi^{\delta}|^{\frac{r-4}{2}}\varphi^{\delta} \cdot (^t \nabla \varphi^{\delta} \cdot w_{\delta}) dx ds,  $$
the H\"older inequality yields
$$
 |I(t)| \leq \int_0^t \|\nabla \varphi^{\delta}(s) |\varphi^{\delta}(s)|^{\frac{r-2}{2}}\|_{L^2(\mathbb{R}^d)} \|(\psi^{(\delta)})^{\frac r4}(s)\|_{L^{\widetilde{q}}(\mathbb{R}^d)}
 \|w_{\delta}(s)\|_{L^q(\mathbb{R}^d)} ds,
$$
where $\widetilde{q}$ is defined by $\frac 12 + \frac 1q + \frac{1}{\widetilde{q}} = 1$.
By the Sobolev embedding $\dot{H}^{1-\frac 2p}(\mathbb{R}^d) \hookrightarrow L^{\widetilde{q}}(\mathbb{R}^d)$, there exists a constant $C = C(p,d)$ such that 
$$\|(\psi^{(\delta)})^{\frac r4}(s)\|_{L^{\widetilde{q}}(\mathbb{R}^d)} \leq C \|(\psi^{(\delta)})^{\frac r4}(s)\|_{\dot{H}^{1-\frac 2p}(\mathbb{R}^d)}. $$
Since $d \geq 3$ and $0 \leq 1- \frac 2p < 1$, we may choose $C$ uniformly in $p$ for fixed $d$.
Interpolating $\dot{H}^{1-\frac 2p}$ between $L^2$ and $\dot{H}^1$ gives
\begin{multline*}
 \|(\psi^{(\delta)})^{\frac r4}(s)\|_{\dot{H}^{1-\frac 2p}(\mathbb{R}^d)} \leq 
\|(\psi^{(\delta)})^{\frac r4}(s)\|_{L^2(\mathbb{R}^d)}^{\frac 2p} \|\nabla (\psi^{(\delta)})^{\frac r4}(s)\|_{L^2(\mathbb{R}^d)}^{1-\frac 2p} \\
= \|\varphi^{\delta}(s)\|_{L^r(\mathbb{R}^d)}^{\frac rp} \|\nabla (\psi^{(\delta)})^{\frac r4}(s)\|_{L^2(\mathbb{R}^d)}^{1-\frac 2p}. 
\end{multline*}
As $\nabla (\psi^{(\delta)})^{\frac r4} = \frac r4  (\psi^{(\delta)})^{\frac{r-4}{4}}\nabla \psi^{(\delta)},$ we may now bound $|I(t)|$ from above by 
$$
 C \int_0^t \|\nabla \varphi^{\delta}(s) |\varphi^{\delta}(s)|^{\frac{r-2}{2}}\|_{L^2(\mathbb{R}^d)} 
 \left(\frac r4  \|(\psi^{(\delta)})^{\frac{r-4}{4}}(s)\nabla \psi^{(\delta)}(s)\|_{L^2(\mathbb{R}^d)}\right)^{1-\frac 2p}
 \|\varphi^{\delta}(s)\|_{L^r(\mathbb{R}^d)}^{\frac rp}
 \|w_{\delta}(s)\|_{L^q(\mathbb{R}^d)} ds.
$$
The Young inequality for real numbers yields, with $\widetilde{p}$ defined in the same way as $\widetilde{q}$,
\begin{multline*}
  |I(t)| \leq  \frac{\nu}{2} \int_0^t \|\nabla \varphi^{\delta}(s) |\varphi^{\delta}(s)|^{\frac{r-2}{2}}\|_{L^2(\mathbb{R}^d)}^2 ds 
 + \frac{r\nu }{4\widetilde{p}} \int_0^t 
    \|(\psi^{(\delta)})^{\frac{r-4}{4}}(s)\nabla \psi^{(\delta)}(s)\|_{L^2(\mathbb{R}^d)}^2 ds \\
    +  \frac{C^p}{p\nu^{p-2}}\int_0^t  \|\varphi^{\delta}(s)\|_{L^r(\mathbb{R}^d)}^r \|w_{\delta}(s)\|_{L^q(\mathbb{R}^d)}^p ds.
\end{multline*}
Absorbing the first two terms in the left-hand side of the inequality gives
$$\frac 1r \|\varphi^{\delta}(t)\|_{L^r(\mathbb{R}^d)}^r
\leq \frac 1r \|\varphi_0\|_{L^r(\mathbb{R}^d)}^r + \frac{C^p}{p\nu^{p-2}}\int_0^t  \|\varphi^{\delta}(s)\|_{L^r(\mathbb{R}^d)}^r \|w_{\delta}(s)\|_{L^q(\mathbb{R}^d)}^p ds. $$
By Gr\"onwall's inequality,
$$\|\varphi^{\delta}(t)\|_{L^r(\mathbb{R}^d)}
\leq  \|\varphi_0\|_{L^r(\mathbb{R}^d)} \exp\left(\frac{C^p}{p\nu^{p-2}}\int_0^t \|w_{\delta}(s)\|_{L^q(\mathbb{R}^d)}^p ds \right). $$
Letting $r$ go to infinity and using the trivial bound $\|w_{\delta}(s)\|_{L^q(\mathbb{R}^d)} \leq \|w(s)\|_{L^q(\mathbb{R}^d)}$ yields
\begin{equation}
 \|\varphi^{\delta}(t)\|_{L^{\infty}(\mathbb{R}^d)}
\leq  \|\varphi_0\|_{L^{\infty}(\mathbb{R}^d)} \exp\left(\frac{C^p}{p\nu^{p-2}}\int_0^t \|w(s)\|_{L^q(\mathbb{R}^d)}^p ds \right). 
\end{equation}
It only remains to take the limit $\delta\to 0$.
As the family $(\varphi^{\delta})_{\delta}$ is a bounded subset in $L^{\infty}(\mathbb{R}_+ \times \mathbb{R}^d)$, up to an extraction, there exists $\varphi$ in 
$L^{\infty}(\mathbb{R}_+ \times \mathbb{R}^d)$ such that 
$$\varphi^{\delta} \rightharpoonup^{\ast} \varphi \qquad  \text{in } L^{\infty}(\mathbb{R}_+ \times \mathbb{R}^d) \text{ as } \delta \to 0. $$
By Fatou's lemma, the bound 
\begin{equation}
 \|\varphi(t)\|_{L^{\infty}(\mathbb{R}^d)}
\leq  \|\varphi_0\|_{L^{\infty}(\mathbb{R}^d)} \exp\left(\frac{C^p}{p\nu^{p-2}}\int_0^t \|w(s)\|_{L^q(\mathbb{R}^d)}^p ds \right)
\end{equation}
follows.
On the other hand, since $v$ and $w$ belong to $L^2(\mathbb{R}_+, \dot{H}^1(\mathbb{R}^d))$, it is clear that
$$v_{\delta}, w_{\delta} \longrightarrow v, w \qquad \text{strongly in } L^2(\mathbb{R}_+, \dot{H}^1(\mathbb{R}^d)) \text{ as } \delta \to 0. $$
Hence, taking the limit $\delta \to 0$ in the equation on $\varphi^{\delta}$, we see that $\varphi$ indeed satisfies the adjoint equation and the proof is over.
 \end{proof}
 We now turn to the proof of the uniqueness theorem.

 \begin{proof}[Proof of Theorem \ref{UniquenessSerrin}]
  Let $\rho = \rho(x)$ be a radial mollifying kernel and define $\rho_{\varepsilon}(x) := \varepsilon^{-d} \rho(\frac{x}{\varepsilon})$.
 Convolving the equation on $a$ by $\rho_{\varepsilon}$ gives, denoting $a_{\varepsilon} := \rho_{\varepsilon} \ast a$,
$$
  (C_{\varepsilon}) \ \ \partial_t a_{\varepsilon} + \nabla \cdot (a_{\varepsilon} \otimes v) - \nu\Delta a_{\varepsilon} 
  = \nabla \cdot (w \otimes a_{\varepsilon}) + C^{\varepsilon} + D^{\varepsilon},
$$
 where the commutator $C^{\varepsilon}$ has been defined in Lemma \ref{LemmeCommutateur}.
 The second commutator is defined by
 $$D^{\varepsilon} :=  \rho_{\varepsilon} \ast \nabla \cdot (a \otimes w) - \nabla \cdot (w \otimes a_{\varepsilon}).$$
 Similarly to what we proved for $C^{\varepsilon}$, we have
 $$\|D^{\varepsilon}\|_{L^1(\mathbb{R}_+ \times \mathbb{R}^d)} \to 0 \text{ as $\varepsilon \to 0$.} $$
Notice that even without any smoothing in time, $a_{\varepsilon}$, $\partial_t a_{\varepsilon}$ lie respectively in $L^{\infty}(\mathbb{R}_+, \mathcal{C}^{\infty}(\mathbb{R}^d))$ and 
$L^1(\mathbb{R}_+, \mathcal{C}^{\infty}(\mathbb{R}^d))$, which is enough to make the upcoming computations rigorous.
In what follows, we let $\varphi^{\delta}$ be a solution of the Cauchy problem $(-C'_{\delta})$, with  $(-C'_{\delta})$ being $(-C')$ where $v$ and $a$ are replaced by $v_{\delta}$ and $a_{\delta}$.
Let us now multiply, for $\delta, \varepsilon > 0$ the equation $(C_{\varepsilon})$ by $\varphi^{\delta}$ and integrate in space and time.
After integrating by parts (which is justified by the high regularity of the terms we have written), we get
$$
 \int_0^T \int_{\mathbb{R}^d} \partial_t a_{\varepsilon}(s,x) \varphi^{\delta}(s,x) dx ds = 
 \langle a_{\varepsilon}(T), \varphi_0 \rangle_{\mathcal{D}'(\mathbb{R}^d), \mathcal{D}(\mathbb{R}^d)}  - 
 \int_0^T \int_{\mathbb{R}^d} a_{\varepsilon}(s,x)  \partial_t \varphi^{\delta}(s,x) dx ds.
$$
From this identity, it follows that
\begin{multline*}
 \langle a_{\varepsilon}(T), \varphi_0 \rangle_{\mathcal{D}'(\mathbb{R}^d), \mathcal{D}(\mathbb{R}^d)}  
 = \int_0^T \int_{\mathbb{R}^d} \varphi^{\delta}(s,x) ( C^{\varepsilon} + D^{\varepsilon})(s,x) dx ds \: +\\  
 \int_0^T \int_{\mathbb{R}^d} a_{\varepsilon}(s,x) 
 \left(\partial_t \varphi^{\delta}(s,x)+ \nabla \cdot (v(s,x)\varphi^{\delta}(s,x))+ \nu\Delta\varphi^{\delta}(s,x) - ^t \nabla \varphi^{\delta}(s,x) \cdot w(s,x) \right) dx ds.
\end{multline*}
From Lemma \ref{LemmeCommutateur}, we know in particular that $C^{\varepsilon}$ belongs to $L^1(\mathbb{R}_+ \times \mathbb{R}^d)$ for each fixed $\varepsilon > 0$ and the same goes for $D^{\varepsilon}$.
Thus, in the limit $\delta \to 0$, we have, for each $\varepsilon > 0$,
$$\int_0^T \int_{\mathbb{R}^d} \varphi^{\delta}(s,x) (C^{\varepsilon} + D^{\varepsilon})(s,x) dx ds \to \int_0^T \int_{\mathbb{R}^d} \varphi(s,x) (C^{\varepsilon} + D^{\varepsilon})(s,x) dx ds. $$
On the other hand, the definition of $\varphi^{\delta}$ gives
$$- \partial_t \varphi^{\delta} - \nabla \cdot (v \varphi^{\delta}) - \nu\Delta \varphi^{\delta} + ^t \nabla \varphi^{\delta} \cdot w
= \nabla \cdot ((v_{\delta} - v) \varphi^{\delta}) + ^t \nabla \varphi^{\delta} \cdot (w - w_{\delta}). $$
Thus, the last integral in the above equation may be rewritten, integrating by parts,
$$- \int_0^T \int_{\mathbb{R}^d} \varphi^{\delta} (v_{\delta} - v)\cdot  \nabla a_{\varepsilon}(s,x) dx ds 
+ \int_0^T \int_{\mathbb{R}^d} \varphi^{\delta}(s,x) \nabla \cdot ( (w-w_{\delta})(s,x) \otimes a_{\varepsilon}(s,x)) dx ds. $$
For each fixed $\varepsilon$, the assumption on $a$ entails that $\nabla a_{\varepsilon}$ belongs to $L^2(\mathbb{R}_+ \times \mathbb{R}^d)$.
Furthermore, it is an easy exercise to show that
$$\|v_{\delta} - v\|_{L^2(\mathbb{R}_+ \times \mathbb{R}^d)} \leq 
\delta \|\nabla v\|_{L^2(\mathbb{R}_+ \times \mathbb{R}^d)} \||\cdot|\rho\|_{L^1(\mathbb{R}^d)} $$
and
$$\|\nabla \cdot ((w - w_{\delta}) \otimes a_{\varepsilon})\|_{L^1(\mathbb{R}_+ \times \mathbb{R}^d)} \to 0 \text{ as } \delta \to 0, \text{ for any fixed } \varepsilon. $$
Now, taking the limit $\delta \to 0$ while keeping $\varepsilon > 0$ fixed, we have
$$
  \langle a_{\varepsilon}(T), \varphi_0 \rangle_{\mathcal{D}'(\mathbb{R}^d), \mathcal{D}(\mathbb{R}^d)} = \int_0^T \int_{\mathbb{R}^d} \varphi(s,x) (C^{\varepsilon} + D^{\varepsilon})(s,x) dx ds .
$$
Taking the limit $\varepsilon \to 0$ and using Lemma \ref{LemmeCommutateur}, we finally obtain
$$
   \langle a(T), \varphi_0 \rangle_{\mathcal{D}'(\mathbb{R}^d), \mathcal{D}(\mathbb{R}^d)} = 0.
$$
This being true for any test function $\varphi_0$, $a(T)$ is the zero distribution and finally $a \equiv 0$.
 \end{proof}

\begin{proof}[Proof of Theorem \ref{NouvellePreuveSerrin}]
 Let $\Omega := \nabla \wedge u$ and $\Omega_0 := \nabla \wedge u_0$. 
 The equation on $\Omega$ writes
$$ 
  (NSV) \left \{
\begin{array}{c  c}
    \partial_t \Omega + \nabla \cdot  (\Omega \otimes u) - \Delta \Omega = \nabla \cdot ( u \otimes \Omega)\\
    \Omega (0)  = \Omega_0.  \\
\end{array}
\right.
$$
 Let $\chi = \chi(t)$ be a smooth cutoff in time supported inside $]T_1,T_2[$.
 Let $\varphi = \varphi(t)$ be another smooth cutoff such that
 $$\text{supp } \chi \subset \{\varphi \equiv 1 \}. $$
 Denoting $\Omega' = \chi \Omega$ and $u' = \varphi u$, we have
$$ 
  (NSV') \left \{
\begin{array}{c  c}
    \partial_t \Omega' + \nabla \cdot  (\Omega' \otimes u) - \Delta \Omega' = \nabla \cdot ( u' \otimes \Omega') + \Omega \partial_t \chi\\
    \Omega' (0)  = 0.  \\
\end{array}
\right.
$$
 Following the same lines as for Theorem \ref{ExistenceSerrin}, we sketch a way to build a solution $\Omega''$ of
 $$ 
 \left \{
\begin{array}{c  c}
    \partial_t \Omega'' + \nabla \cdot  (\Omega'' \otimes u) - \Delta \Omega'' = \nabla \cdot ( u' \otimes \Omega'') + \Omega \partial_t \chi\\
    \Omega'' (0)  = 0.  \\
\end{array}
\right.
$$
belonging to 
$$L^{\infty}(\mathbb{R}_+, L^2(\mathbb{T}^3)) \cap L^2(\mathbb{R}_+, \dot{H}^1(\mathbb{T}^3)). $$
 For $\delta > 0$, let $u_{\delta}, u'_{\delta}$ and $\Omega_{\delta}$ be smooth space mollifications of $u,u'$ and $\Omega$ respectively.
 By the Friedrichs method and heat kernel estimates, there exists a smooth solution $\Omega''_{\delta}$ of 
 $$ 
  \left \{
\begin{array}{c  c}
    \partial_t \Omega''_{\delta} + \nabla \cdot  (\Omega''_{\delta} \otimes u_{\delta}) - \Delta \Omega''_{\delta} 
    = \nabla \cdot ( u'_{\delta} \otimes \Omega''_{\delta}) + \Omega_{\delta} \partial_t \chi\\
    \Omega''_{\delta} (0)  = 0.  \\
\end{array}
\right.
$$
Performing an energy estimate in $L^2(\mathbb{T}^3)$ gives
\begin{multline*}
 \frac 12  \|\Omega''_{\delta}(t)\|_{L^2(\mathbb{T}^3)}^2 +  \int_0^t \|\nabla \Omega''_{\delta}(s)\|_{L^2(\mathbb{T}^3)}^2 ds \\ 
 = \int_0^t \int_{\mathbb{T}^3} \Omega''_{\delta}(x,s) \cdot 
\left( \nabla \cdot \left(u'_{\delta}(x,s) \otimes \Omega''_{\delta}(x,s) \right) + \Omega_{\delta}(x,s) \partial_t \chi(s) \right) dx ds.
\end{multline*}
The right-hand side decomposes in two terms, which we estimate separately.
For the first one, we integrate by parts and use H\"older inequality to get
\begin{multline*}
  \int_0^t \int_{\mathbb{T}^3} \Omega''_{\delta}(x,s) \cdot 
\left( \nabla \cdot \left(u'_{\delta}(x,s) \otimes \Omega''_{\delta}(x,s) \right) \right) dx ds \\
= - \int_0^t \int_{\mathbb{T}^3} \nabla \Omega''_{\delta}(x,s) :  
\left(u'_{\delta}(x,s) \otimes \Omega''_{\delta}(x,s) \right)  dx ds \\
\leq \int_0^t \|\nabla \Omega''_{\delta}(s) \|_{L^2(\mathbb{T}^3)} \| u'_{\delta}(s)\|_{L^q(\mathbb{T}^3)}
\|\Omega''_{\delta}(s)\|_{L^{\tilde{q}}(\mathbb{T}^3)} ds,
\end{multline*}
where $\tilde{q}$ is defined by 
$$\frac{1}{\tilde{q}} = \frac 12 - \frac 1q. $$
The Sobolev embedding $\dot{H}^{\frac 3q}(\mathbb{T}^3) \hookrightarrow L^{\tilde{q}}(\mathbb{T}^3)$ gives
$$
 \|\Omega''_{\delta}(s)\|_{L^{\tilde{q}}(\mathbb{T}^3)}
\lesssim \|\nabla \Omega''_{\delta}(s)\|_{L^2(\mathbb{T}^3)}^{\frac 3q} \|\Omega''_{\delta}(s)\|_{L^2(\mathbb{T}^3)}^{\frac 2p}.
$$
Hence,
\begin{multline*}
  \int_0^t \int_{\mathbb{T}^3} \Omega''_{\delta}(x,s) \cdot 
\left( \nabla \cdot \left(u'_{\delta}(x,s) \otimes \Omega''_{\delta}(x,s) \right) \right) dx ds \\
\lesssim \int_0^t \|\nabla \Omega''_{\delta}(s)\|_{L^2(\mathbb{T}^3)}^{1+\frac 3q} 
\|u'_{\delta}(s)\|_{L^q(\mathbb{T}^3)} \|\Omega''_{\delta}(s)\|_{L^2(\mathbb{T}^3)}^{\frac 2p} ds.
\end{multline*}
Young inequality entails the existence of a constant $C$ depending only on $q$ such that 
\begin{multline*}
 \|\nabla \Omega''_{\delta}(s)\|_{L^2(\mathbb{T}^3)}^{1+\frac 3q} 
\| u'_{\delta}(s)\|_{L^q(\mathbb{T}^3)} \|\Omega''_{\delta}(s)\|_{L^2(\mathbb{T}^3)}^{\frac 2p} \\
\leq \frac 12 \|\nabla \Omega''_{\delta}(s) \|_{L^2(\mathbb{T}^3)}^2
 + \frac C2 \| u'_{\delta}(s)\|_{L^q(\mathbb{T}^3)}^p \|\Omega''_{\delta}(s)\|_{L^2(\mathbb{T}^3)}^2
\end{multline*}
The second term is easier to bound.
Indeed, thanks to the Cauchy-Schwarz inequality,
\begin{multline*}
  \int_0^t \int_{\mathbb{T}^3} \Omega''_{\delta}(x,s) \cdot \Omega_{\delta}(x,s) \partial_t \chi(s) ds 
\leq \int_0^t \|\Omega''_{\delta}(s)\|_{L^2(\mathbb{T}^3)} \|\Omega_{\delta}(s)\|_{L^2(\mathbb{T}^3)} |\partial_t \chi(s)| ds \\
\leq \frac 12 \int_0^t \|\Omega''_{\delta}(s)\|_{L^2(\mathbb{T}^3)}^2 ds 
+ \frac 12 \|\partial_t \chi\|_{L^{\infty}(\mathbb{R}_+)}^2\int_0^t \|\Omega_{\delta}(s)\|_{L^2(\mathbb{T}^3)}^2  ds.
\end{multline*}
Gathering these estimates, we have shown that, for some constant $C$ depending only on $q$,
\begin{multline*}
 \|\Omega''_{\delta}(t)\|_{L^2(\mathbb{T}^3)}^2 +  \int_0^t \|\nabla \Omega''_{\delta}(s)\|_{L^2(\mathbb{T}^3)}^2 ds \\ 
 \leq \int_0^t  \left(1+  C \| u'_{\delta}(s)\|_{L^q(\mathbb{T}^3)}^p\right)  \|\Omega''_{\delta}(s)\|_{L^2(\mathbb{T}^3)}^2 ds
 + \|\partial_t \chi\|_{L^{\infty}(\mathbb{R}_+)}^2 \|\Omega_{\delta}\|_{L^2(\mathbb{R}_+ \times \mathbb{T}^3)}^2.
\end{multline*}
Since $u'_{\delta}$ and $\Omega_{\delta}$ are mollifications of $u'$ and $\Omega$ respectively, for any $s \in \mathbb{R}_+$ and any $\delta > 0$, there holds
$$\|u'_{\delta}(s)\|_{L^q(\mathbb{T}^3)} \leq \|u'(s)\|_{L^q(\mathbb{T}^3)}$$
and
$$\|\Omega_{\delta}(s)\|_{L^2(\mathbb{T}^3)} \leq \|\Omega(s)\|_{L^2(\mathbb{T}^3)}.$$
Combining these facts to Gr\"onwall's inequality entails the bound
\begin{multline*}
 \|\Omega''_{\delta}(t)\|_{L^2(\mathbb{T}^3)}^2 +  \int_0^t \|\nabla \Omega''_{\delta}(s)\|_{L^2(\mathbb{T}^3)}^2 ds \\ 
 \leq \|\partial_t \chi\|_{L^{\infty}(\mathbb{R}_+)}^2 \|\Omega\|_{L^2(\mathbb{R}_+ \times \mathbb{T}^3)}^2
 \exp\left(t+ C\int_0^t  \| u'(s)\|_{L^q(\mathbb{T}^3)}^p ds \right).
\end{multline*}
It only remains to pass to the limit.
From the uniform $L^{\infty}(\mathbb{R}_+, L^2(\mathbb{T}^3)) \cap L^2(\mathbb{R}_+, \dot{H}^1(\mathbb{T}^3))$ bound on the $(\Omega''_{\delta})_{\delta}$, there exists some
$\Omega'' \in L^{\infty}(\mathbb{R}_+, L^2(\mathbb{T}^3)) \cap L^2(\mathbb{R}_+, \dot{H}^1(\mathbb{T}^3))$ such that, up to an extraction,
$$\Omega''_{\delta} \rightharpoonup \Omega''\text{ weakly in } L^{\infty}(\mathbb{R}_+, L^2(\mathbb{T}^3)) \cap L^2(\mathbb{R}_+, \dot{H}^1(\mathbb{T}^3)) \text{ as } \delta \to 0 $$
This weak convergence allows us to pass to the limit in the equation on $\Omega''_{\delta}$, thanks to the strong convergences
$$u_{\delta}, u'_{\delta} \to u, u' \text{ strongly in } L^2(\mathbb{R}_+, \dot{H}^1(\mathbb{T}^3)) \text{ as } \delta \to 0. $$
Such an $\Omega''$ thus belongs to 
$$L^{\infty}(\mathbb{R}_+, L^2(\mathbb{T}^3)) \cap L^2(\mathbb{R}_+, \dot{H}^1(\mathbb{T}^3))$$
and solves, as required,
 $$ 
 \left \{
\begin{array}{c  c}
    \partial_t \Omega'' + \nabla \cdot  (\Omega'' \otimes u) - \Delta \Omega'' = \nabla \cdot ( u' \otimes \Omega'') + \Omega \partial_t \chi\\
    \Omega'' (0)  = 0.  \\
\end{array}
\right.
$$
 Now, letting $\widetilde{\Omega} := \Omega' - \Omega''$, we see that $\widetilde{\Omega}$ solves
$$
  (NSV^0) \left \{
\begin{array}{c  c}
    \partial_t \widetilde{\Omega} + \nabla \cdot  (\widetilde{\Omega} \otimes u) - \Delta \widetilde{\Omega} = \nabla \cdot ( u' \otimes \widetilde{\Omega})\\
    \widetilde{\Omega} (0)  = 0.  \\
\end{array}
\right.
$$
 We recall that $u$ and $u'$ belong to $L^2(\mathbb{R}_+, \dot{H}^1(\mathbb{T}^3))$ and by assumption, $u'$ further belongs to $L^p(\mathbb{R}_+, L^q(\mathbb{T}^d))$.
 Moreover, the high regularity of $\Omega''$ and the fact that $u$ is a Leray solution of the Navier-Stokes equations together entail that $\widetilde{\Omega}$ belongs to 
 $L^2(\mathbb{R}_+ \times \mathbb{T}^3)$.
 These regularity assumptions allow us to invoke Theorem \ref{UniquenessSerrin}, from which we deduce that $\widetilde{\Omega} \equiv 0$.
 It follows that
 $$\Omega \in  L^{\infty}_{loc}(]T_1,T_2[, L^2(\mathbb{T}^3)) \cap L^2_{loc}(]T_1,T_2[, \dot{H}^1(\mathbb{T}^3)). $$ 
 Since $u$ is the inverse curl of $\Omega$, the above regularity on $\Omega$ is equivalent to
 $$u\in  L^{\infty}_{loc}(]T_1,T_2[, \dot{H}^1(\mathbb{T}^3)) \cap L^2_{loc}(]T_1,T_2[, \dot{H}^2(\mathbb{T}^3)). $$ 
 From here, improving again the regularity on $\Omega$ and $u$ relies on an induction procedure, which is tedious to write thoroughly but not difficult.
 We need to prove that, for all $s \in \mathbb{N}$, we have
 $$\Omega \in L^{\infty}_{loc}(]T_1,T_2[, \dot{H}^s(\mathbb{T}^3)) \cap L^2_{loc}(]T_1,T_2[, \dot{H}^{s+1}(\mathbb{T}^3)), $$
 which is equivalent to requiring 
 $$u\in  L^{\infty}_{loc}(]T_1,T_2[, \dot{H}^{s+1}(\mathbb{T}^3)) \cap L^2_{loc}(]T_1,T_2[, \dot{H}^{s+2}(\mathbb{T}^3)). $$ 
 The case $s = 0$ is exactly what we just proved.
 To go from the step $s$ to the step $s+1$, we simply compute all the space derivatives of order $s+1$ of the equation on $\Omega'$.
 More precisely, denoting by $\partial^{s+1}$ a generic space derivative of order $s+1$, we have
 $$\partial_t \partial^{s+1} \Omega'  + \nabla \cdot (u \otimes \partial^{s+1} \Omega') - \Delta \partial^{s+1} \Omega'
 = \nabla (\partial^{s+1}\Omega' \otimes u) + (\text{l.o.t in } \Omega').$$
 Performing an energy estimate in $L^2(\mathbb{T}^3)$ as above and using Theorem \ref{UniquenessSerrin}, we get
 $$\partial^{s+1}\Omega'  \in L^{\infty}_{loc}(]T_1,T_2[, L^2(\mathbb{T}^3)) \cap L^2_{loc}(]T_1,T_2[, \dot{H}^1(\mathbb{T}^3)),$$
 which is what we wanted.
 Time derivatives may now be handled by a similar induction argument, which we will not write.
 This closes the proof.
 \end{proof}

  \begin{proof}[Proof of Theorem \ref{BorneL1VorticiteForte}]
   Given the assumptions we made, we compute the vorticity equation by taking the curl on each side of the Navoer-Stokes equations. 
   Let $s < t$ be two real numbers in $]0,T[$.
   Let $\varphi : [0,t-s] \times \mathbb{R}^d \to \mathbb{R}^d$ be a solution of the adjoint equation satisfying the bound \eqref{GrowingMaximumPriciple}.
   Imitating the proof of Theorem \ref{UniquenessSerrin} for the time interval $[s,t]$, we arrive at 
   \begin{equation}
    \langle \Omega(t), \varphi_0 \rangle_{L^1(\mathbb{R}^d), L^{\infty}(\mathbb{R}^d)} = \langle \Omega(s), \varphi(t-s)\rangle_{L^1(\mathbb{R}^d), L^{\infty}(\mathbb{R}^d)}.
   \end{equation}
Thanks to the bound \eqref{GrowingMaximumPriciple}, for any $\varphi_0$ in $\mathcal{D}(\mathbb{R}^d)$, we have 
\begin{equation}
 |\langle \Omega(t), \varphi_0 \rangle_{L^1(\mathbb{R}^d), L^{\infty}(\mathbb{R}^d)}| \leq 
 \|\Omega(s)\|_{L^1(\mathbb{R}^d)}\|\varphi_0\|_{L^{\infty}(\mathbb{R}^d)} \exp\left[\frac{C^p}{p\nu^{p-2}} \int_s^t \|u(s)\|_{L^q(\mathbb{R}^d)}^p ds \right].
\end{equation}
Taking the supremum over all possible $\varphi_0$ and letting $s \to 0$ yields the result.
  \end{proof}

\subsection*{Acknowledgements}

The author is grateful to L. Sz\'ekelyihidi both for his kind invitation to Universit\"at Leipzig and the fruitful discussions which led to notable improvements of the paper.
\textbf{TODO : virer l'approximation et ne montrer qu'une estimation a priori pour l'existence}

\end{document}